\documentclass[12pt, reqno]{amsart}

\usepackage{amscd}        
\usepackage{float}
\usepackage{enumitem}
\usepackage{nomencl}
\usepackage{url}
\usepackage{algorithm}
\usepackage{algpseudocode}
\usepackage{tabularx}

\usepackage[utf8]{inputenc}
\usepackage{setspace}
\usepackage{mathtools, nccmath}
\usepackage{tabu}

\usepackage{multirow}

\usepackage{tikz-cd}
\usepackage[utf8]{inputenc}
\usepackage{mathtools}
\usepackage[utf8]{inputenc}
\usepackage{amsmath}
\usepackage{amssymb}
\usepackage{amsthm}
\usepackage[inner=1.0in,outer=1.0in,bottom=1.0in, top=1.0in]{geometry}

\newtheorem{theorem}{Theorem}[section]

\newtheorem{lemma}[theorem]{Lemma}

\newtheorem{corollary}[theorem]{Corollary}

\theoremstyle{definition}
\newtheorem{remark}[theorem]{Remark}

\begin{document}

\title{Asymptotic Distribution of the Partition Crank}

\author{Asimina Hamakiotes, Aaron Kriegman, and Wei-Lun Tsai}

\address{Department of Mathematics, University of Connecticut, Storrs, CT 06269-1009, United States of America}
\email{asimina.hamakiotes@uconn.edu}
\address{Department of Pure Mathematics and Mathematical Statistics (Jesus College), University of Cambridge, Cambridge, United Kingdom}
\email{ak2313@cam.ac.uk}
\address{Department of Mathematics, University of Virginia, Charlottesville, VA 22904, United States of America}
\email{wt8zj@virginia.edu}

\thanks{This work was supported by NSF grant DMS-1757872.}

\maketitle



\begin{abstract}
Freeman Dyson conjectured the existence of an unknown partition statistic he called the crank which would explain Ramanujan's partition congruence mod 11 just as his rank statistic explains Ramanujan's partition congruences mod 5 and 7. Such a crank statistic was found by Andrews and Garvan in 1988.
In this paper, we investigate the 
crank counting function, which counts the number of partitions of $n$ with crank congruent to $r$ mod $Q$. First, we obtain an effective bound on the error 
term in Zapata Rol\'on's asymptotic formula for the crank counting function. We then use this to prove that the crank counting function 
is asymptotically equidistributed mod $Q$, for any odd number $Q$. We also use this to study surjectivity of the crank when viewed as a function 
from partitions to the integers mod $Q$, and to prove strict log-subadditivity of the crank counting function. The latter result is analogous to Bessenrodt and Ono's 
strict log-subadditivity of the partition function.  
\end{abstract}
\vspace{0.1in}
\small{\hspace{0.3in}\textit{Keywords}: Crank; Effective equidistribution; Log-subadditivity; Partition.\\
\vspace{0.001in}
\hspace{0.39in}\textit{Mathematics Subject Classification $(2020)$}: 05A17; 11P82; 11P83.}
\normalsize


\section{Introduction and Statement of Results}
                           
A \textit{partition} of a positive integer $n$ is a non-increasing sequence of positive integers 
$\lambda_1 \geq \lambda_2 \geq \dots \geq \lambda_{\ell} > 0$, called its \textit{parts}, such that $\lambda_1 + \lambda_2 + \cdots + \lambda_{\ell}=n$. 
Let $p(n)$ count the number of partitions of $n$. 
In 1918, Hardy and Ramanujan \cite{hardy1918asymptotic} gave the following asymptotic formula for $p(n)$:
\[p(n) \sim \frac{1}{4\sqrt{3}n}e^{\pi \sqrt{\frac{2n}{3}}}\]
as $n\rightarrow\infty$.

Ramanujan \cite{Ra1,Ra2} also proved the following famous congruences for the partition function:
for any $l\in\mathbb{Z}_{\geq0}$ we have
\begin{align*}
    p(5l + 4) &\equiv 0 \pmod5,\\
    p(7l + 5) &\equiv 0 \pmod7,\\
    p(11l + 6) &\equiv 0 \pmod{11}.
\end{align*}

Atkin and Watson \cite{Atkin,Wat} proved generalizations of Ramanujan's congruences modulo any integer of the form $5^a7^b11^c$, 
where $a,b,c \in \mathbb{N}$. Ono \cite{Ono} proved that there are infinitely many new partition congruences for any prime modulus $Q\geq5$.

Dyson \cite{Dyson} conjectured that Ramanujan's congruences modulo $5$ and $7$ could be explained using a 
function he called the \textit{rank}. The rank of a partition $\lambda$ is defined to be its largest part minus the number of its parts; namely, 
\[\text{rank}(\lambda) := \lambda_1 - \ell.\]

Let $N(r,Q;n)$ be the number of partitions of $n$ with rank congruent to $r$ modulo $Q$. Dyson conjectured that:
\begin{itemize}
    \item For each $r\pmod5$, 
    \begin{align*}
        N(r,5;5l+4) = \frac{p(5l+4)}{5}.
    \end{align*}
    \item For each $r\pmod7$, 
    \begin{align*}
        N(r,5;7l+5) = \frac{p(7l+5)}{7}.
    \end{align*}
\end{itemize}
In 1954, Atkin and Swinnerton-Dyer \cite{atkin1954some} proved this conjecture. 

Dyson observed that the rank fails to explain Ramanujan's congruence modulo $11$. He instead conjectured the existence of another statistic 
which he called the \textit{crank} which would explain all three congruences. In 1988, Andrews and Garvan \cite{andrews1988dyson} found such a crank. 
More precisely, let $o(\lambda)$ be the number of $1$'s in $\lambda$ and $\nu (\lambda)$ be the number of parts of $\lambda$ larger than $o(\lambda)$. 
The crank of $\lambda$ is then defined to be
\[\text{crank}(\lambda) := 
\begin{cases}
\lambda_1 & \text{if } o(\lambda)=0 \\
\nu(\lambda) - o(\lambda) & \text{if } o(\lambda)>0.
\end{cases}\]

\begin{remark} 
Garvan, Kim and Stanton \cite{GKS} found different cranks which also explain all of Ramanujan's congruences.
\end{remark}

Let $M(r,Q;n)$ be the number of partitions of $n$ with crank $r$ modulo $Q$. Ramanujan's congruences follow from the ``exact'' 
equidistribution of $M(r,Q;n)$ on the residue classes $r\pmod Q$ for certain $Q$ and $n$. Here we show that $M(r,Q;n)$ becomes equidistributed on the residue classes $r \pmod Q$ for odd $Q$ as $n\to\infty$. More precisely,
Zapata Rol\'on \cite{rolon2013asymptotische} gave an asymptotic formula for $M(r,Q;n)$ with an error term which is 
$O(n^\epsilon)$. Here we refine his analysis to give an effective bound on the error term with explicit constants. We then use this bound to prove the following 
effective equidistribution theorem. 

Let $\mu(n):=\sqrt{24n-1}$. 

\begin{theorem}\label{asymptoticmain}
Let $0 \leq r < Q$ with $Q$ an odd integer. Then we have
\[\frac{M(r,Q;n)}{p(n)} = \frac{1}{Q} + R(r,Q;n),\]
where when $Q<11$ we have
\[\left|R(r,Q;n)\right| \leq 10^5(40.93 Q + 6.292) e^{-\left(1-\frac{1}{Q}\right)\frac{\pi\mu(n)}{6}}n^\frac{11}{8}\]
and when $Q\geq 11$ we have
\[\left|R(r,Q;n)\right| \leq 10^5(40.93 Q + 6.292) e^{-\left(1-\sqrt{1 +12(\frac{1}{Q^2}-\frac{1}{Q})}\right)\frac{\pi\mu(n)}{6}}n^\frac{11}{8}.\]
\end{theorem}

It follows immediately that the cranks are asymptotically equidistributed modulo $Q$. 
\begin{corollary}\label{uniform}
Let $0 \leq r < Q$ with $Q$ an odd integer. Then we have
\[\frac{M(r,Q;n)}{p(n)} ~\longrightarrow ~\frac{1}{Q}\]
as $n \rightarrow \infty$. 
\end{corollary}
In the recent work \cite{M}, Masri proved a quantitative equidistribution theorem for partition ranks (mod $2$) with a power-saving error term by using spectral methods and subconvexity bounds. 

Corollary \ref{uniform} can be seen as analogous to Dirichlet's theorem on the equidistribution of primes among the residue classes $r \pmod Q$ with $(r,Q)=1$. 
Motivated by this, we will use Theorem \ref{asymptoticmain} to prove an analog of Linnik's theorem which gives an upper bound 
for the smallest prime in each residue class $r \pmod Q$.

\begin{theorem}\label{surject} Let $Q$ be an odd integer and for 
$Q\geq11$ we define the constant
\begin{align}\label{constant}
    C_Q:=\frac{(1.93\times10^{59})(40.93Q^2+6.292Q)^8}{\left(\pi-\pi\sqrt{1+12(\frac{1}{Q^2}-\frac{1}{Q})}\right)^{24}}+1.
\end{align}
Then we have
\begin{align*}
M(r,Q;n) > 0
\end{align*}
if $Q<11$ and $n\geq263$,
or if $Q\geq11$ and $n\geq C_Q$.

\end{theorem}


We will also prove the following result (which includes the case that $Q$ is even) using a different, combinatorial argument.

\begin{theorem}\label{exact} For odd $Q \geq 11$ we have 
\begin{align*}
    M(r,Q;n)>0
\end{align*}
if and only if $n\geq\frac{Q+1}{2}$. For even $Q\geq8$ we have
\begin{align*}
    M(r,Q;n)>0
\end{align*}
if and only if $n\geq\frac{Q}{2}+2$.
\end{theorem}

In a related direction, Bessenrodt and Ono \cite{bessenrodt2016maximal} proved strict log-subadditivity of the partition function. Dawsey and Masri \cite{LM} later proved strict log-subadditivity for the Andrews spt-function. We will use Theorem \ref{asymptoticmain} 
to prove strict log-subadditivity of the crank counting function. 

\begin{theorem}\label{subadditive}
Given any residue $r\pmod Q$ where $Q$ is odd, we have 
\begin{align*}
    M(r,Q;a+b) < M(r,Q;a)M(r,Q;b),
\end{align*}
if $Q<11$ and $a,b \geq396$ or if $Q\geq11$ and $a,b\geq C_Q$, where $C_Q$ is defined in $(\ref{constant})$.
\end{theorem}


\vspace{0.10in}

\textbf{Acknowledgements}. We would like to thank Professor Riad Masri for his invaluable help and guidance on this work. We also thank the referees, Professor Krishnaswami Alladi, and Professor Frank Garvan for very helpful comments and corrections on the manuscript.

\section{Effective Asymptotic Formula for $M(r,Q;n)$}

In \cite{rolon2013asymptotische, Rolon}, Zapata Rol\'on gives an asymptotic formula for $M(r,Q;n)$. Here we refine his analysis and 
give an asymptotic formula with an effective bound on the error term. We begin by stating a few necessary definitions. 

Let
\[\omega_{h,k} := \exp(\pi is(h,k)),\]
where the Dedekind sum $s(h,k)$ is defined by
\[s(h,k) := \sum_{u \pmod k} \left(\left( \frac{u}{k} \right)\right) \left(\left( \frac{hu}{k} \right)\right).\]
Here $((\cdot))$ is the sawtooth function defined by
\[((x)) := 
\begin{cases}
x - \lfloor x \rfloor - \frac{1}{2} & $if $ x\in \mathbb{R}\backslash \mathbb{Z}, \\
0 & $if $ x\in \mathbb{Z}.
\end{cases}\]

Let $0\leq h<k$ be relatively prime integers. Let $0<r<Q$ be relatively prime integers where $Q$ is odd. Let $h'$ be a solution to the congruence $hh'\equiv-1 \pmod{k}$ if $k$ is odd and $hh'\equiv-1\pmod{2k}$ if $k$ is even. Let $c_1:=\frac{c}{(c,k)}$ and $k_1:=\frac{k}{(c,k)}$. Let $l$ be the minimal positive solution to $l\equiv ak_1\pmod{c_1}$. For $m,n\in \mathbb{Z}$ we define:
\[\widetilde{B}_{a,c,k}(n,m) := (-1)^{ak+1}\sin\left(\frac{\pi a}{c}\right) \sum_{\substack{h \pmod k \\ (h,k)=1}} \frac{\omega_{h,k}}{\sin\left(\frac{\pi ah'}{c}\right)}e^{\frac{-\pi ia^2k_1h'}{c}}e^{\frac{2\pi i}{k}(nh + mh')},\]
where the sum runs over all primitive residue classes modulo $k$.

For the case $c \nmid k$ we define: 
\[D_{a,c,k}(m,n) := (-1)^{ak+l}\sin\left(\frac{\pi a}{c}\right) \sum_{\substack{h \pmod k \\ (h,k)=1}}\omega_{h,k}e^{\frac{2\pi i}{k}(nh + mh')},\]
where $l$ is the solution to $l \equiv ak_1$ (mod $c_1$).

In order to provide certain bounds, we define the following:
\[\delta_{a,c,k,r}^i := 
\begin{cases} 
-(\frac{1}{2} + r)\frac{l}{c_1} + \frac{1}{2}(\frac{l}{c_1})^2 + \frac{1}{24} & $if $i = +,\\
\frac{l}{2c_1} + \frac{1}{2}(\frac{l}{c_1})^2 - \frac{23}{24} - r(1 - \frac{l}{c_1}) & $if $i = -,
\end{cases}\]
\[\delta_0:=\frac{1}{2Q^2}-\frac{1}{2Q}+\frac{1}{24}<\frac{1}{24},\]
and
\[m_{a,c,k,r}^+ := \frac{1}{2c_1^2}(-a^2k_1^2 + 2lak_1 - ak_1c_1 - l^2 + lc_1 - 2ark_1c_1 + 2lc_1r),\]
\[m_{a,c,k,r}^- := \frac{1}{2c_1^2}(-a^2k_1^2 + 2lak_1 - ak_1c_1 - l^2 + 2c_1^2r - 2lrc_1 + 2ark_1c_1 + 2lc_1 + 2c_1^2 - ak_1c_1).\]
Note that $\delta^\pm_{a,Q,k,r}\leq\delta_0<\frac{1}{24}$. \\

Zapata Rol\'on obtains an asymptotic formula for $M(r,Q;n)$ by using the circle method. First, he defines the generating function
\begin{equation*}
    C(w,q):=\sum_{n=0}^\infty\sum_{m=-\infty}^\infty M(m,n)w^mq^n,
\end{equation*}
where $M(m,n)$ is the number of partitions of $n$ with crank $m$. In order to use the modular properties of this function, he plugs in a root of unity for $w$ and studies the coefficients of $q$. Additionally, he defines
\begin{equation*}
C(e^{2\pi i\frac{j}{k}},q):=\sum_{n=0}^\infty\tilde{A}\left(\frac{j}{k},n\right)q^n,
\end{equation*}
and uses the circle method to find an asymptotic formula for  $\tilde{A}\left(\frac{j}{k},n\right)$, and uses the identity
\begin{equation}\label{flip}
M(r,Q;n)=\frac{1}{Q}\sum_{j=0}^{Q-1}\zeta^{-jr}\tilde{A}\left(\frac{j}{Q},n\right)
\end{equation}
to get an asymptotic formula for $M(r,Q;n)$. 
Note that $\tilde{A}(\frac{0}{Q},n)=p(n)$. 

Moreover, Zapata Rol\'on 
gives the following asymptotic formula for $\tilde{A}\left(\frac{j}{Q},n\right)$:
\begin{equation*}
    \tilde{A}\left(\frac{j}{Q},n\right) =  \frac{4\sqrt{3i}}{\mu(n)}\sum_{\substack{Q|k\\k\leq\sqrt n}}\frac{\Tilde{B}_{j,Q,k}(-n,0)}{\sqrt{k}}\sinh\left(\frac{\pi\mu(n)}{6k}\right)
\end{equation*}
\begin{equation*}
    +\frac{8\sqrt{3}\sin(\frac{\pi j}{Q})}{\mu(n)} \sum_{\substack{k,s \\ Q\nmid k \\ \delta^i_{j,Q,k,s}>0 \\ i \in \{+,-\}}}
\frac{D_{j,Q,k}(-n,m^{i}_{j,Q,k,s})}{\sqrt{k}}\sinh\left(\sqrt{24\delta^i_{j,Q,k,s}}\frac{\pi\mu(n)}{6k}\right)+O(n^\epsilon),
\end{equation*}
which when plugged into equation (\ref{flip}) gives
\begin{equation*}
    M(r,Q;n) = \frac{1}{Q}p(n)+\frac{1}{Q}\sum_{j=1}^{Q-1}\zeta^{-rj}_{Q} \frac{4\sqrt{3i}}{\mu(n)}
\sum_{\substack{Q|k\\k\leq\sqrt n}}\frac{\Tilde{B}_{j,Q,k}(-n,0)}{\sqrt{k}}\sinh\left(\frac{\pi\mu(n)}{6k}\right)
\end{equation*}
\begin{equation*}
    +\frac{1}{Q}\sum_{j=1}^{Q-1}\zeta^{-rj}_{Q}\frac{8\sqrt{3}\sin(\frac{\pi j}{Q})}{\mu(n)} \sum_{\substack{k,s \\ Q\nmid k \\ \delta^i_{j,Q,k,s}>0 \\ i \in \{+,-\}}}\frac{D_{j,Q,k}(-n,m^{i}_{j,Q,k,s})}{\sqrt{k}}\sinh\left(\sqrt{24\delta^i_{j,Q,k,s}}\frac{\pi\mu(n)}{6k}\right)+O(n^\epsilon).
\end{equation*}









\vspace{0.10in}

\textbf{Proof of Theorem \ref{asymptoticmain}}. 
We first break the $O(n^\epsilon)$ error term from the calculation of $\tilde{A}(\frac{j}{Q},n)$ into six pieces: $S_{err},S_{1err},S_{2err},T_{err},$ and the 
contributions of error from certain integrals which we will call $\Sigma_1I_{err}$ and $\Sigma_2I_{err}$. Zapata Rol\'on provides bounds on each of those pieces, which we can then refine and sum up to get bounds on the error in the formula for 
$\tilde{A}(\frac{j}{Q},n)$. Then using equation (\ref{flip}) and the triangle inequality, we can get our desired bound on $|R(r,Q;n)|$.

Fix odd integers $j$ and $Q$. We will bound the error coming from $\tilde{A}(\frac{j}{Q},n)$. Zapata Rol\'on provides the following bounds:
\begin{gather*}
|S_{err}| \leq \frac{2e^{2\pi + \frac{\pi}{24}}|\sin(\frac{\pi j}{Q})|(c_2 + 2(1 + |\cos(\frac{\pi}{Q})|)c_1(1 + c_2))n^{\frac{1}{4}}\left(1 + \log\left(\frac{Q-1}{2}\right)\right)}{\pi(1 - \frac{\pi ^2}{24})Q},\\
|T_{err}| \leq 16e^{2\pi}f(Q)n^{\frac{1}{4}}\left|\sin\left(\frac{\pi j}{Q}\right)\right|,
\end{gather*}
where
\begin{equation*}
f(Q) := \frac{1+c_2e^{\pi\delta_0}}{1-e^{\frac{-\pi}{Q}}} + e^{\pi\delta_0}c_1(1+c_2) + \frac{e^{\pi\delta_0}(c_2+1)c_3}{2},
\end{equation*}
and where the $c_i$ are constants defined in \cite{rolon2013asymptotische}. We have the approximations $c_1 \leq 0.046$, $c_2 \leq 1.048$, and $c_3 \leq 0.001$.
Also, 
\begin{gather*}
|S_{1err}| \leq \frac{8e^{2\pi + \frac{\pi}{12}}(1 + \log(\frac{Q-1}{2}))n^{\frac{1}{4}}}{\pi(1 - \frac{\pi^2}{24})Q},\\
|S_{2err}| \leq 32e^{2\pi}n^{\frac{1}{4}}\left|\sin\left(\frac{\pi j}{Q}\right)\right|\frac{e^{2\pi\delta_0}}{1 - e^{\frac{-2\pi}{Q}}},\\
|\Sigma_1I_{err}| \leq \frac{4\left(\frac{4}{3}+2^\frac{5}{4}\right)\left|\sin\left(\frac{\pi j}{Q}\right)\right|\left(1+\log\left(\frac{Q-1}{2}\right)\right)e^{2\pi+\frac{\pi}{12}}n^\frac{3}{8}}{\pi\left(1-\frac{\pi^2}{24}\right)Q},\\
|\Sigma_2I_{err}| \leq 8\left(\frac{4}{3}+2^\frac{5}{4}\right)\left|\sin\left(\frac{\pi j}{Q}\right)\right|\frac{e^{2\pi\delta_0+2\pi}}{1-e^{-\frac{2\pi}{Q}}}.
\end{gather*}


Now we estimate some of the expressions in those bounds in order to simplify them:
\begin{enumerate}[label=$\bullet$]
 \item $\left|\sin\left(\frac{\pi j}{Q}\right)\right|\leq 1,$
 \vspace{0.1in}
    \item $\frac{(1 + \log(\frac{Q-1}{2}))}{\pi(1 - \frac{\pi^2}{24})Q}\leq0.1902,$
    \vspace{0.1in}
    \item $\left(\frac{4}{3}+2^\frac{5}{4}\right)\leq3.712,$
    \vspace{0.1in}
    \item $\frac{1}{1-e^{-\frac{\pi}{Q}}}\leq\pi Q,$
    \vspace{0.1in}
    \item $\frac{1}{1-e^{-\frac{2\pi}{Q}}}\leq 2\pi Q.$
\end{enumerate}

In order to prove these last two bounds, let $g(x):=\frac{1}{1-e^{-\frac{b}{x}}},$ where $b,x>0$. Then we have $g'(x)=b\frac{g(x)^2}{x^2}e^{-\frac{b}{x}}$. Moreover, the function $h(x):=bx$ satisfies $h'(x)=b$, 
and in the case of $b=\pi$ and $b=2\pi$ we have $h(1)>g(1)$ and $h'(1)>g'(1)$ by a short calculation. This implies that $h(Q)>g(Q)$ for all $Q\geq1$, as desired.

Also, we simplify the bounds given in \cite{rolon2013asymptotische}:
\begin{itemize}
\item$|S_{err}| \leq 330.9n^{\frac{1}{4}}$,
\vspace{0.1in}
\item$|T_{err}| \leq (59071 Q +930.05)n^{\frac{1}{4}}$,
\vspace{0.1in}
\item$|S_{1err}| \leq 1059n^{\frac{1}{4}}$,
\vspace{0.1in}
\item$|S_{2err}| \leq 22306n^{\frac{1}{4}}$,
\vspace{0.1in}
\item$|\Sigma_1I_{err}| \leq 1965n^\frac{3}{8}$,
\vspace{0.1in}
\item$|\Sigma_2I_{err}| \leq 113883 Q$.
\end{itemize}
Summing these all up gives the total contribution of the $O(n^\epsilon)$ error term to $\tilde{A}(\frac{j}{Q},n)$. We then use equation (\ref{flip}) to get the contribution of the error term to $M(r,Q;n)$. However, after applying the triangle inequality these two bounds will be the same except for a factor of ${(Q-1)}/{Q}$, which we will round up to 1 for simplicity. So, the bound for the $O(n^\epsilon)$ error term of $M(r,Q;n)$ is 
\[(172954 Q + 26591) n^\frac{3}{8}.\]

Now we will bound the main terms from the formula for $M(r,Q;n)$ by using the following bounds from \cite{rolon2013asymptotische}:
\begin{itemize}
\item $\widetilde{B}_{j,Q,k}(-n,0) \leq \frac{2k\left(1 +\log\frac{Q-1}{2} \right)}{\pi \left(1 -\frac{\pi^2}{24}\right)}\leq 0.3804kQ,$
\vspace{0.1in}
\item $D_{j,Q,k}(-n,m^{i}_{j,Q,k,s})\leq k,$
\vspace{0.1in}
\item $\sum\limits_{\substack{Q|k\\k\leq\sqrt{n}}}k^{\frac{1}{2}} \leq \frac{2}{3Q}n^{\frac{3}{4}},$
\vspace{0.1in}
\item $\sinh\left(\sqrt{24\delta_{j,Q,k,s}^j}\frac{\pi\mu(n)}{6k}\right) \leq \frac{1}{2}e^{\sqrt{24\delta_0}\frac{\pi\mu(n)}{6k}}.$
\end{itemize}

First, we have the following estimate.
\[\sum\limits_{\substack{Q\mid k\\k\leq\sqrt{n}}}k^{\frac{1}{2}}\leq\int_{0}^{\sqrt{n}}x^{\frac{1}{2}}dx =\frac{2}{3}n^{\frac{3}{4}}.\]

Hence, we obtain
\begin{align*}
\left|\frac{1}{Q}\sum_{j=1}^{Q-1}\zeta^{-rj}_{Q} \frac{4\sqrt{3i}}{\mu(n)}\sum_{\substack{Q|k\\k\leq\sqrt n}}\frac{\Tilde{B}_{j,Q,k}(-n,0)}{\sqrt{k}}\sinh\left(\frac{\pi\mu(n)}{6k}\right)\right|
&\leq\frac{4\sqrt{3}}{\mu(n)}\sinh\left(\frac{\pi\mu(n)}{6Q}\right)\sum_{\substack{Q|k\\k\leq\sqrt n}}0.3804k^{\frac{1}{2}} Q\\
&\leq 1.757\frac{1}{\mu(n)}\frac{1}{2}e^{\frac{\pi\mu(n)}{6Q}}n^\frac{3}{4}\\
&\leq 0.8785e^{\frac{\pi\mu(n)}{6Q}}n^\frac{1}{4}.
\end{align*}

Next, by the same argument in \cite[page 35]{rolon2013asymptotische},  it follows that for fixed $k$ we have
\[\sum_{\substack{k,s \\ Q\nmid k \\ \delta^i_{j,Q,k,s}>0 \\ i \in \{+,-\}}}1\leq \frac{Q+18}{24}.\]
Similarly, we bound the other main term:
\begin{align*}
&\left|\frac{1}{Q}\sum_{j=1}^{Q-1}\zeta^{-rj}_{Q}\frac{8\sqrt{3}\sin(\frac{\pi j}{Q})}{\mu(n)} \sum_{\substack{k,s \\ Q\nmid k \\ \delta^i_{j,Q,k,s}>0 \\ i \in \{+,-\}}}\frac{D_{j,Q,k}(-n,m^{i}_{j,Q,k,s})}{\sqrt{k}}\sinh\left(\sqrt{24\delta^i_{j,Q,k,s}}\frac{\pi\mu(n)}{6k}\right)\right|\\
&\leq \frac{8\sqrt{3}}{\mu(n)} \sinh\left(\sqrt{24\delta_0}\frac{\pi\mu(n)}{6}\right)\sum_{\substack{k,s \\ Q\nmid k \\ \delta^i_{j,Q,k,s}>0 \\ i \in \{+,-\}}}{k}^{\frac{1}{2}}\\
&\leq \frac{8\sqrt{3}}{\mu(n)} \frac{1}{2}e^{\sqrt{24\delta_0}\frac{\pi\mu(n)}{6}}\frac{Q+18}{24}\cdot\frac{2}{3}n^\frac{3}{4}\\
&\leq (0.1924Q+3.464)e^{\sqrt{24\delta_0}\frac{\pi\mu(n)}{6}}n^\frac{1}{4}.
\end{align*}

From \cite{lehmer1939remainders} we get the following lower bound for $p(n)$:
\[p(n) > \frac{\sqrt{3}}{12n}\left(1-\frac{1}{\sqrt{n}}\right)e^{\frac{\pi\mu(n)}{6}}.\]
We also note that for $n\geq2$, $$\frac{1}{1 - \frac{1}{\sqrt{n}}} \leq \frac{1}{1 - \frac{1}{\sqrt{2}}} \leq 3.415.$$

Finally, combining all the estimates, we get
\begin{align*}
&|R(r,Q;n)|
=\left|\frac{M(r,Q;n)}{p(n)} -\frac{1}{Q}\right|\\
&\leq\left|\frac{1}{p(n)} 0.8785e^{\frac{\pi\mu(n)}{6Q}}n^\frac{1}{4} + \frac{1}{p(n)} (0.1924Q+3.464)e^{\sqrt{24\delta_0}\frac{\pi\mu(n)}{6}}n^\frac{1}{4} + \frac{1}{p(n)} (172954 Q + 26591) n^\frac{3}{8}\right|\\
&\leq 20.79e^{\left(\frac{1}{Q}-1\right)\frac{\pi\mu(n)}{6}}n^\frac{5}{4} +  (4.553Q+81.96)e^{\left(\sqrt{24\delta_0}-1\right)\frac{\pi\mu(n)}{6}}n^\frac{5}{4} + 10^5(40.93 Q + 6.292) e^{-\frac{\pi\mu(n)}{6}}n^\frac{11}{8}.
\end{align*}
This is a sum of three terms each with similar factors. In order to combine this into an upper bound which can be worked with we take the sum of all three coefficients, the highest order exponential, and the highest power of $n$ from the three terms and put them together in one term. This gives the bounds in the statement of the theorem. We have to break up the $Q<11$ and $Q\geq11$ cases because that is the point at which ${1}/{Q}-1$ is overtaken by $\sqrt{24\delta_0}-1$. Note that the third term has far larger coefficients but also a much faster decaying exponential term, so a lot of accuracy is lost when combining this term with the others.
\qed


\section{Surjectivity}

The crank is a function that maps the set of partitions $S_n$ of $n$ to the integers $\mathbb{Z}$. We can 
take the reduction of this map modulo $Q$ to get a function from $S_n$ to $\mathbb{Z}/Q\mathbb{Z}$. 
It is natural to ask for which $n$ this map is surjective. This is an analogue of Linnik's theorem on the smallest prime in an 
arithmetic progression. We study this question in Theorems \ref{surject} and \ref{exact}, which we now prove in turn.

\vspace{0.10in}

\textbf{Proof of Theorem \ref{surject}}. In order to prove that the reduction of the crank modulo $Q$ is surjective, it is sufficient to prove that 
\[|R(r,Q;n)| < \frac{1}{Q},\]
because this implies $M(r,Q;n)>0$.

By our bounds on $|R(r,Q;n)|$, when $Q<11$ we need
\[10^5(40.93 Q + 6.292) e^{-\left(1-\frac{1}{Q}\right)\frac{\pi}{6}\mu(n)}n^\frac{11}{8} < \frac{1}{Q},\]
and when $Q\geq 11$ we need \[10^5(40.93 Q + 6.292) e^{-\left(1-\sqrt{1+12(\frac{1}{Q^2}-\frac{1}{Q})}\right)\frac{\pi}{6}\mu(n)}n^\frac{11}{8} < \frac{1}{Q}.\]

First assume that $Q < 11$. Then in order to show the inequality 
\begin{align*}
    &10^5(40.93 Q + 6.292) e^{-\left(1-\frac{1}{Q}\right)\frac{\pi}{6}\mu(n)}n^\frac{11}{8} < \frac{1}{Q},
\end{align*}
it suffices to show that
\begin{align}\label{sur1}
   10^5(40.93\times11 + 6.292) e^{-\left(1-\frac{1}{3}\right)\frac{\pi}{6}\mu(n)}n^\frac{11}{8}  <\frac{1}{11}.
\end{align}
By a short computation, we find that (\ref{sur1}) holds when $n\geq263$. 

Hence, it follows that if $Q<11$ and $n\geq263$, then
\begin{align*}
    |R(r,Q;n)| < \frac{1}{Q}.
\end{align*}

Next, we deal with the case $Q\geq11$. It suffices to show that  
\begin{align*}
    |R(r,Q;n)|\leq10^5(40.93 Q + 6.292) e^{-\left(1-\sqrt{1+12(\frac{1}{Q^2}-\frac{1}{Q})}\right)\frac{\pi}{6}\mu(n)}n^\frac{11}{8} < \frac{1}{2Q},
\end{align*}
where we replaced $1/Q$ with $1/2Q$ since we will need this inequality in Section \ref{SLS}.
To verify the inequality, it is equivalent to show that
\begin{align}\label{midinequality}
    \frac{e^{\left(1-\sqrt{1+12(\frac{1}{Q^2}-\frac{1}{Q})}\right)\frac{\pi}{6}\mu(n)}}{n^{\frac{11}{8}}}>2\times10^5Q(40.93 Q + 6.292).
\end{align}
Moreover, we recall the following inequality \cite[Eq. 4.5.13]{DLMF}
\begin{align}\label{exp}
    e^x>\left(1+\frac{x}{y}\right)^y,\qquad x,y>0.
\end{align}
Hence, by taking $y=3$ in (\ref{exp}), we get
\begin{align*}
  \frac{e^{\left(1-\sqrt{1+12(\frac{1}{Q^2}-\frac{1}{Q})}\right)\frac{\pi}{6}\mu(n)}}{n^{\frac{11}{8}}}>&
  \frac{1}{n^{\frac{11}{8}}}\left(1+\left(1-\sqrt{1+12(\frac{1}{Q^2}-\frac{1}{Q})}\right)\frac{\pi}{18}\mu(n)\right)^3\\
  >& \frac{\pi^3(24n-1)^{\frac{3}{2}}}{18^3n^{\frac{11}{8}}}\left(1-\sqrt{1+12(\frac{1}{Q^2}-\frac{1}{Q})}\right)^3.
\end{align*}
By combining (\ref{midinequality}), it suffices to show that
\begin{align}\label{midinequality2}
    \frac{(24n-1)^{\frac{3}{2}}}{n^{\frac{11}{8}}}>\frac{2\times10^5\times18^3Q(40.93 Q + 6.292)}{\left(\pi-\pi\sqrt{1+12(\frac{1}{Q^2}-\frac{1}{Q})}\right)^3}.
\end{align}
Also, if $n\geq2$, then we have
\begin{align*}
   \frac{(24n-1)^{\frac{3}{2}}}{n^{\frac{11}{8}}}>\frac{24^{\frac{3}{2}}(n-1)^{\frac{3}{2}}}{n^{\frac{11}{8}}}=24^{\frac{3}{2}}\left(1-\frac{1}{n}\right)^{\frac{11}{8}}(n-1)^{\frac{1}{8}}\geq\frac{24^{\frac{3}{2}}}{2^{\frac{11}{8}}}(n-1)^{\frac{1}{8}}.
\end{align*}
Hence, by a simple calculation, if we choose the constant
\begin{align*}
    C_Q:=\frac{(1.93\times10^{59})(40.93Q^2+6.292Q)^8}{\left(\pi-\pi\sqrt{1+12(\frac{1}{Q^2}-\frac{1}{Q})}\right)^{24}}+1,
\end{align*}
then (\ref{midinequality2}) holds when $n\geq C_Q>2$. This completes the proof.
\qed

\begin{remark}
From our estimation, the exponent $y$ in $(\ref{exp})$ controls the magnitude of $C_Q$. Hence, it is not hard to see that we can choose the constant 
$C_Q$ so that $C_Q\asymp Q$ for $y$ sufficiently large.
\end{remark}


\vspace{0.10in}

There is a different, combinatorial method which allows us to include the case that $Q$ is even.

We will need the following lemma.

\begin{lemma}\label{lemmaA}
For $n\geq 6$, the cranks of the partitions of $n$ take on exactly the values $-n$ through $n$ except for $-n+1$ and $n-1$.
\end{lemma}

\begin{proof}
It is clear from the definition of crank that a partition $\lambda$ of $n$ cannot have crank larger than $n$, 
since $\lambda_1 \leq n$ and $\nu(\lambda)$ is much less than $n$. The crank cannot be less than $-n$ since $o(\lambda)\leq n$. 
Say there was a partition $\lambda$ with crank $n-1$. Since $\nu(\lambda)$ is much less than $n$, it must be that $o(\lambda)=0$ 
and therefore $\lambda_1 = n-1$, but this implies $\lambda_2=1$, which is a contradiction. 
Now say we have a partition $\lambda$ with crank $-n+1$. If every part of $\lambda$ is 1, then the crank would be $-n$, 
so we must have $\lambda_1\geq 2$. This implies $o(\lambda)\leq n-2$, so the crank can not be $-n+1$.

We have shown that the crank can only take on the claimed values. We now show that it takes on each of those values. 
Let $3\leq k\leq n$ and we will construct a partition $\lambda$ of crank $k$. Let $\lambda_1$ be $k$. 
If $n-k$ is even, then let all the remaining parts be 2. If $n-k$ is odd, then let $\lambda_2$ be 3 and let all the remaining parts be 2. 
Notice that this does not work when $k=n-1$ because 1 cannot be written as a sum of $2s$ and $3s$. 
We can also create a partition of crank $-k$ by letting there be $k$ 1's, and letting the remaining parts be 2 or 3 as before. 
Since $k\geq 3$, we have $\nu(\lambda)=0$, and so the partition has the desired crank. Note that once again this does not work 
when $k=n-1$ for the same reason as before. Now it only remains to find partitions with cranks equal to 2, 1, 0, $-1$, and $-2$. For $n\geq 7$, 
the following partitions work:
\begin{enumerate}[label=$\bullet$]
    \item $n = (n-5) + 2 + 2 + 1$,
    \item $n = (n-3) + 2 + 1$,
    \item $n = (n-1) + 1$,
    \item $n = (n-2) + 1 + 1$,
    \item $n = (n-3) + 1 + 1 + 1$.
\end{enumerate}
For $n=6$, we also must consider the partitions $2+2+2$ and $2+2+1+1$ with cranks $2$ and $-2$ respectively, and for the 1, 0, and $-1$ cases the above partitions still work.
\end{proof}

\vspace{0.10in}

\textbf{Proof of Theorem \ref{exact}}.
For even $Q$ and $n\geq{Q}/{2}+2$, Lemma \ref{lemmaA} implies that the crank takes on at least $Q$ consecutive values, so the crank maps onto each residue class. For $n={Q}/{2}+1$, no partition has crank congruent to ${Q}/{2}$. For $n={Q}/{2}$, no partition has crank congruent to ${Q}/{2}-1$. For lower $n$, no partition has crank congruent to ${Q}/{2}$.

For odd $Q$ and $n={(Q+1)}/{2}$, the residues ${(Q\pm1)}/{2}$ are mapped onto by $-n$ and $n$, and all the other residues are mapped onto by $-n+2$ through $n-2$. For $n>{(Q+1)}/2$, the crank takes on at least $Q$ consecutive values. When $n = {(Q-1)}/{2}$, no partition has crank congruent to ${(Q-3)}/{2}$. For lower $n$, no partition has crank congruent to ${(Q-1)}/{2}$. Thus we have shown that for odd $Q\geq11$ and even $Q\geq8$, the cranks of the partitions of $n$ take on every value modulo $Q$ exactly when $n\geq {(Q+1)}/{2}$ or $n\geq{Q}/{2}+2$ respectively, as desired.
\qed


\section{Strict log-Subadditivity for Crank Functions}\label{SLS}
Bessenrodt and Ono \cite{bessenrodt2016maximal} showed that if  $a,b\geq 1$ and $a+b\geq 9$, then
\begin{align*}
    p(a+b)<p(a)p(b).
\end{align*}
Also, Dawsey and Masri \cite{LM} showed the following similar result for the spt-function,
\begin{align*}
     \textrm{spt}(a+b)<\textrm{spt}(a)\textrm{spt}(b),
\end{align*}
for $(a,b)\neq(2,2)$ or $(3,3)$.

We now prove Theorem \ref{subadditive}, which is an analogous result 
for the crank counting function.

\vspace{0.10in}

\textbf{Proof of Theorem \ref{subadditive}}. 
We first deal with the case $Q<11$. By our bounds on $|R(r,Q;n)|$, when $Q<11$ we have

\begin{align*}
L(Q,n)<M(r,Q;n)<U(Q,n),
\end{align*}
where
\begin{align*}
    L(Q,n)&:=p(n)\left( \frac{1}{Q}-10^5(40.93 Q + 6.292) e^{-\left(1-\frac{1}{Q}\right)\frac{\pi}{6}\mu(n)}n^\frac{11}{8}\right),\\
    U(Q,n)&:=p(n)\left( \frac{1}{Q}+10^5(40.93 Q + 6.292) e^{-\left(1-\frac{1}{Q}\right)\frac{\pi}{6}\mu(n)}n^\frac{11}{8}\right).
\end{align*}
Moreover, by $3\leq Q<11$ and $n\geq263$, we have
\begin{align*}
     L(Q,n)>p(n)\left( \frac{1}{11}-10^5(40.93\times 11 + 6.292) e^{-\frac{\pi}{9}\mu(n)}n^\frac{11}{8}\right)>(0.00306)p(n).
\end{align*}
Similarly, we get
\begin{align*}
    U(Q,n)<p(n)\left( \frac{1}{3}+10^5(40.93\times11 + 6.292) e^{-\frac{\pi}{9}\mu(n)}n^\frac{11}{8}\right)<(1.10213\times10^7)p(n).
\end{align*}
Hence, if $n\geq263$, then we have
\begin{align*}
 (0.00306)p(n)<M(r,Q;n)<(1.10213\times10^7)p(n).
\end{align*}
On the other hand, Lehmer \cite{lehmer1939remainders} gives the following bounds for $p(n)$:
\begin{align*} 
    \frac{\sqrt{3}}{12n}\left(1-\frac{1}{\sqrt{n}}\right)e^{\frac{\pi}{6}\mu(n)}<p(n)< \frac{\sqrt{3}}{12n}\left(1+\frac{1}{\sqrt{n}}\right)e^{\frac{\pi}{6}\mu(n)}.
\end{align*}
Together these give the bounds
\begin{align}\label{keybound}
    (0.00306) \frac{\sqrt{3}}{12n}\left(1-\frac{1}{\sqrt{n}}\right)e^{\frac{\pi}{6}\mu(n)}<M(r,Q;n)
    <(1.10213\times10^7)\frac{\sqrt{3}}{12n}\left(1+\frac{1}{\sqrt{n}}\right)e^{\frac{\pi}{6}\mu(n)}.
\end{align}
Now, we follow the argument in \cite[Section 6]{LM} and let $b=Ca$ for some $C\geq1$. Then by (\ref{keybound}), it follows that
\begin{align*}
    M(r,Q,a)M(r,Q,b)>(0.00306)^2\frac{1}{48Ca^2}\left(1-\frac{1}{\sqrt{a}}\right)\left(1-\frac{1}{\sqrt{Ca}}\right)e^{\frac{\pi}{6}(\mu(a)+\mu(Ca))},
\end{align*}
and
\begin{align*}
    M(r,Q,a+b)<(1.10213\times10^7)\frac{\sqrt{3}}{12(a+Ca)}\left(1+\frac{1}{\sqrt{a+Ca}}\right)e^{\frac{\pi}{6}\mu(a+Ca)}.
\end{align*}
It suffices to show that
\begin{align*}
   T_a(C)>\log\left(V_a(C)\right)+\log\left(S_a(C)\right),
\end{align*}
where
\begin{align*}
T_a(C)&:=\frac{\pi}{6}(\mu(a)+\mu(Ca)-\mu(a+Ca)),\\
    S_a(C)&:=\frac{1+\frac{1}{\sqrt{a+Ca}}}{\left(1-\frac{1}{\sqrt{a}}\right)\left(1-\frac{1}{\sqrt{Ca}}\right)},\\
V_a(C)&:=\frac{(1.10213\times10^7)}{(0.00306)^2}\frac{4\sqrt{3}Ca}{C+1}.
\end{align*}
As functions of $C$, it can be shown that $T_a(C)$ is increasing and $S_a(C)$ is decreasing for
$C \geq 1$, and by combining 
\begin{align*}
    V_a(C)<\frac{(1.10213\times10^7)}{(0.00306)^2}4\sqrt{3}a,
\end{align*}
it suffices to show that
\begin{align}\label{TSV}
    T_a(1)=\frac{\pi}{6}(2\mu(a)-\mu(2a))&>\log\left(\frac{(1.10213\times10^7)}{(0.00306)^2}4\sqrt{3}a\right)+\log\left(\frac{1+\frac{1}{\sqrt{2a}}}{\left(1-\frac{1}{\sqrt{a}}\right)^2}\right)\\\notag
    &=\log\left(\frac{(1.10213\times10^7)}{(0.00306)^2}4\sqrt{3}a\right)+\log\left(S_a(1)\right).
\end{align}
By computing the values $T_a(1)$ and $S_a(1)$, we find that (\ref{TSV}) holds for all $a \geq 396$. 

Hence, if $Q<11$ and $a,b\geq396$, then we have
\begin{align*}
    M(r,Q;a+b) < M(r,Q;a)M(r,Q;b).
\end{align*}

Next, we deal with the case $Q\geq11$.  By our bounds on $|R(r,Q;n)|$, when $Q\geq11$ we have

\begin{align*}
L_2(Q,n)<M(r,Q;n)<U_2(Q,n),
\end{align*}
where
\begin{align*}
    L_2(Q,n)&:=p(n)\left( \frac{1}{Q}-10^5(40.93 Q + 6.292) e^{-\left(1-\sqrt{1 +12(\frac{1}{Q^2}-\frac{1}{Q})}\right)\frac{\pi}{6}\mu(n)}n^\frac{11}{8}\right),\\
    U_2(Q,n)&:=p(n)\left( \frac{1}{Q}+10^5(40.93 Q + 6.292) e^{-\left(1-\sqrt{1 +12(\frac{1}{Q^2}-\frac{1}{Q})}\right)\frac{\pi}{6}\mu(n)}n^\frac{11}{8}\right).
\end{align*}
By the proof of Theorem \ref{surject}, we know that if $n\geq C_Q$, then we have 
\begin{align*}
    |R(r,Q;n)|\leq10^5(40.93 Q + 6.292) e^{-\left(1-\sqrt{1 +12(\frac{1}{Q^2}-\frac{1}{Q})}\right)\frac{\pi}{6}\mu(n)}n^\frac{11}{8}<\frac{1}{2Q}.
\end{align*}
It follows that 
\begin{align*}
    \left(\frac{1}{2Q}\right)p(n)<M(r,Q;n)<\left(\frac{3}{2Q}\right)p(n).
\end{align*}
By the same argument of the case $Q<11$, we need to show that for any $b=Ca$ for some $C\geq1$,
\begin{align*}
   T_a(C)>\log\left(W_a(C)\right)+\log\left(S_a(C)\right),
\end{align*}
where
\begin{align*}
W_a(C):=(6Q)\frac{4\sqrt{3}Ca}{C+1}.
\end{align*}
Moreover, by the trivial bound
\begin{align*}
    W_a(C)<24\sqrt{3}Qa,
\end{align*}
and the same argument, it suffices to show that 
\begin{align}\label{TVW2}
    T_a(1)=\frac{\pi}{6}(2\mu(a)-\mu(2a))&>\log(24\sqrt{3}Qa)+\log\left(\frac{1+\frac{1}{\sqrt{2a}}}{\left(1-\frac{1}{\sqrt{a}}\right)^2}\right)\\\notag
    &=\log(24\sqrt{3}Qa)+\log\left(S_a(1)\right).
\end{align}
On the other hand, if $a\geq2$, then we get
\begin{align}\label{TVW3}
    \log(24\sqrt{3}Qa)+\log\left(S_a(1)\right)<\log(24\sqrt{3}Qa)+\log\left(\frac{1+\frac{1}{\sqrt{4}}}{\left(1-\frac{1}{\sqrt{2}}\right)^2}\right)<\log(432Qa).
\end{align}
Also, if $a\geq (432Q)^2\geq (432\times11)^2$, then we have
\begin{align}\label{TVW4}
      T_a(1)=\frac{\pi}{6}\frac{16a-1}{\sqrt{48a-1}}>2\log a\geq\log a + 2\log432Q>\log(432Qa).
\end{align}
Hence, by combining (\ref{TVW2}), (\ref{TVW3}) and (\ref{TVW4}),
we can choose $a,b\geq (432Q)^2$ to get the desired result.
\qed

\end{document}